\documentclass[11pt,leqno,a4paper,twoside,final]{article}

\usepackage[T1]{fontenc}

\usepackage{exscale,ifthen,amsfonts,amsmath,amscd,amssymb,amsthm}

\usepackage[square]{natbib}

\allowdisplaybreaks[2]

\DeclareMathOperator{\cl}{cl}

\DeclareMathOperator{\Prb}{\mathbf{P}}
\DeclareMathOperator{\Mean}{\mathbf{E}}

\theoremstyle{definition}
\newtheorem{example}{Example}[section]
\theoremstyle{plain}

\newtheorem{lemma}{Lemma}[section]

\newtheorem{proposition}{Proposition}[section]
\newtheorem{remark}{Remark}[section]

\newtheorem{theorem}{Theorem}[section]

\theoremstyle{remark}

\numberwithin{equation}{section}

\newcommand{\s}{\sigma}
\newcommand{\E}{\mathbf{E}}
\newcommand{\PR}{\mathbb{P}}
\newcommand{\N}{\mathbb{N}}
\newcommand{\R}{\mathbb{R}}

\newcommand{\B}{\mathcal{B}}
\newcommand{\F}{\mathcal{F}}

\newcommand{\trans}[1]{\ensuremath{{#1}^\mathsf{T}}}

\newcommand{\bcdot}{\ensuremath{\boldsymbol{\cdot}}}

\renewcommand{\phi}{\varphi}
\renewcommand{\epsilon}{\varepsilon}

\begin{document}

% Titel, Danksagung
\title{On large deviations for small noise It{\^o} processes}

\author{Alberto Chiarini\thanks{
Department of Mathematics, Technische Universit{\"a}t Berlin, Stra{\ss}e des 17.~Juni 136, 10623 Berlin, Germany. Email: chiarini@math.tu-berlin.de}
\ and Markus Fischer\thanks{
Department of Mathematics, Universit{\`a} di Padova, via Trieste 63, 35121 Padova, Italy. Email: fischer@math.unipd.it}
}

\date{}

\maketitle

\begin{abstract}
	The large deviation principle in the small noise limit is derived for solutions of possibly degenerate It{\^o} stochastic differential equations with predictable coefficients, which may depend also on the large deviation parameter. The result is established under mild assumptions using the Dupuis-Ellis weak convergence approach. Applications to certain systems with memory and to positive diffusions with square-root-like dispersion coefficient are included.
\end{abstract}



\section{Introduction}
\label{SectIntro}

Freidlin-Wentzell estimates for It{\^o} stochastic differential equations of diffusion type are concerned with large (order one) deviations of solutions to  
\begin{equation} \label{EqIntroSDE}
	dX^{\epsilon}_t = b(X^{\epsilon}_t)\,dt + \sqrt{\epsilon}\,\s(X^\epsilon_t)\,dW_t
\end{equation}
from their small noise limit as the noise parameter $\epsilon > 0$ tends to zero. The small noise limit here is the deterministic dynamical system given by the ordinary differential equation
\begin{equation} \label{EqIntroODE}
	d\phi_t = b(\phi_t)dt.
\end{equation}
In \eqref{EqIntroSDE} and \eqref{EqIntroODE} above, the solutions are $\R^d$-valued, $b$ is a vector field $\R^d \rightarrow \R^d$, $\s$ a matrix-valued function $\R^d \rightarrow \R^{d\times m}$, and $W$ an $m$-dimensional standard Brownian motion, which serves as a model for noise. Solutions of \eqref{EqIntroSDE} and \eqref{EqIntroODE} are usually considered over a finite time interval, say $[0,T]$, with the same deterministic initial condition $X^{\epsilon}_0 = x = \phi_0$.

Large deviations are quantified in terms of the large deviation principle; see, for instance, Section~1.2 in \citet{dembozeitouni}. Let us recall the definition in the context of Polish spaces (i.e., topological spaces that are separable and compatible with a complete metric). Let $\mathcal{X}$ be a Polish space. A \emph{rate function} on $\mathcal{X}$ is a lower semicontinuous function $\mathcal{X}\rightarrow [0,\infty]$. A rate function is said to be \emph{good} if its sublevel sets are compact. The \emph{large deviation principle} is said to hold for a family $(\xi^{\epsilon})_{\epsilon>0}$ of $\mathcal{X}$-valued random variables with rate function $I$ if for all $\Gamma\in \mathcal{B}(\mathcal{X})$,
\begin{align*}
	-\inf_{x\in \Gamma^{\circ}} I(x) &\leq \liminf_{\epsilon\to 0+} \epsilon\log\Prb\left( \xi^{\epsilon}\in \Gamma \right)	\\
	&\leq \limsup_{\epsilon\to 0+} \epsilon\log\Prb\left( \xi^{\epsilon}\in \Gamma \right) \leq -\inf_{x\in \cl(\Gamma)} I(x),
\end{align*}
where $\cl(\Gamma)$ denotes the closure and $\Gamma^{\circ}$ the interior of $\Gamma$. We will also need the following alternative characterization. The \emph{Laplace principle} is said to hold for a family $(\xi^{\epsilon})_{\epsilon>0}$ of $\mathcal{X}$-valued random variables with rate function $I$ if for all $F\in \mathbf{C}_b(\mathcal{X})$ (i.e., $F$ bounded and continuous),
\[
	\lim_{\epsilon\to 0+} -\epsilon\log\Mean\left[\exp\left(-\frac{1}{\epsilon}F(\xi^{\epsilon}) \right) \right] = \inf_{x\in\mathcal{X}}\left\{I(x) + F(x)\right\}.
\]
If the rate function $I$ is good, then the Laplace principle holds with rate function $I$ if and only if the large deviation principle holds with rate function $I$; see, for instance, Section~1.2 in \citet{dupuisellis}.

Various sets of assumptions on the coefficients $b$, $\s$ in \eqref{EqIntroSDE} are known to imply that the large deviation principle holds for the family $(X^\epsilon)_{\epsilon > 0}$ of $\mathbf{C}([0,T],\R^d)$-valued random variables. In the non-degenerate case, that is, if $d = m$ and the matrix-valued function $\s\trans{\s}$ is uniformly positive definite, the large deviation principle holds if, for instance, $b$ and $\s$ are bounded and uniformly continuous; the rate function then takes the form
\[
	I_x(\phi) =\frac{1}{2}\int_0^T \trans{(\dot{\phi_s}-b(s,\phi_s))} (\s\trans{\s})^{-1}(s,\phi_s)(\dot{\phi_s}-b(s,\phi_s))\,ds
\]
whenever $\phi\in \mathbf{C}([0,T],\R^d)$ is absolutely continuous with $\phi_0=x$, and $I_x(\phi) = \infty$ otherwise; see Theorem~5.3.1 in \citet[pp.\,154-155]{freidlinwentzell}. In the general case of possibly degenerate diffusion matrix, the rate function $I_x$ can be expressed as
\begin{equation} \label{EqIntroRateFnct}
	I_x(\phi) \doteq \inf_{\{f\in L^{2}: \phi = x + \int_{0}^{\bcdot}(b(\phi_s) + \s(\phi_s)f_s)ds\}} \frac{1}{2} \int_{0}^{T} |f_t|^2 dt,
\end{equation}
where $\inf\emptyset = \infty$ by convention; see, for instance, Section~5.6 of \citet{dembozeitouni}, where $b$, $\sigma$ are assumed to be globally Lipschitz continuous and $\sigma$ is bounded. In \citet{baldicaramellino}, building on an older work by \citet{baldichaleyatmaurel88}, which in turn improves on results obtained by \citet{priouret82}, the large deviation principle with rate function given by \eqref{EqIntroRateFnct} is established for locally Lipschitz continuous coefficients $b$, $\sigma$ satisfying a sublinear growth condition. The result in \citet{baldicaramellino} is actually more general, see the discussion in Section~\ref{SectAppl} below. The three works just mentioned all use a method of proof due to \citet{azencott80}. The idea is to show that when $\sqrt{\epsilon} W$ is close to a path $\psi = \int_{0}^{.}f_{t}dt$, where $f\in L^{2}([0,T],\mathbb{R}^{m})$, then the probabilities that $X^{\epsilon}$ deviates from the solution $\phi$ to the the integral equation
\begin{equation} \label{EqIntroControlIntEq}
	\phi_{t} = x + \int_{0}^{t} b(\phi_{s})ds + \int_{0}^{t}\sigma(\phi_{s})f_{s}ds,\quad t\in [0,T],
\end{equation}
are exponentially small in $\epsilon$. This can be interpreted as a quasi-continuity property of the It{\^o} solution map associated with $b$, $\sigma$. To verify the quasi-continuity property, assuming that Equation~\eqref{EqIntroControlIntEq} is well-posed given any ``control'' $f\in L^{2}([0,T],\mathbb{R}^{m})$, one first establishes, using a discretization argument applied to Equation~\eqref{EqIntroSDE} and exponential martingale inequalities, the quasi-continuity property for the zero control and time dependent drift coefficients; the estimate is then transferred to controls in $L^2$ and the original coefficients using a change of measure based on Girsanov's theorem. 

In this paper, we study small noise large deviations for possibly degenerate It{\^o} stochastic differential equations with coefficients $b$, $\s$ that may depend on time and the past of the solution trajectory (predictable coefficients) as well as on the large deviation parameter $\epsilon$; cf.\ Equation~\eqref{SDE2} below. This general setting has also been studied in \citet{puhalskii}. The proof of the large deviation principle there is based on Puhalskii's weak convergence approach to large deviations, which builds on idempotent probability theory and convergence in terms of maxingale problems, the idempotent analogues of martingale problems; see \citet{puhalskiiBook}. The assumptions needed in \citet{puhalskii} to establish the large deviation principle are very mild, the main assumption being that Luzin weak uniqueness holds for the idempotent It{\^o} stochastic differential equation associated with the predictable coefficients $b$, $\s$; sufficient conditions in terms of regularity and growth properties of $b$, $\s$ are provided.

The approach we follow here in establishing the large deviation principle, actually through the Laplace principle, is the weak convergence approach introduced by \citet{dupuisellis} and adapted to the study of stochastic systems driven by finite-dimensional Brownian motion in \citet{bouedupuis}. The approach, or more precisely the variational formula for Laplace functionals which is its starting point, has been extended to stochastic systems driven by infinite-dimensional Brownian motion and/or a Poisson random measure in \citet{budhirajadupuis} and \citet{budhirajaInfinite, budhirajaJumps}. Using that approach in the present situation, it is straightforward to prove the large deviation principle for solutions of Equation~\eqref{EqIntroSDE} when the coefficients are globally Lipschitz continuous; see Section~4.2 in \citet{bouedupuis} or, for the case of finite-dimensional jump diffusions, Section~4.1 in \citet{budhirajaJumps}. Here, we obtain the large deviation principle for predictable coefficients under much weaker hypotheses, which can be summarized as follows: continuity of the coefficients in the state variable; strong existence and uniqueness for the (stochastic) prelimit equations; uniqueness for a controlled version of the (deterministic) limit equation; stability of the prelimit solutions under $L^2$-bounded perturbations in terms of tightness of laws. An advantage of the weak convergence method is that the large deviation principle can be derived in a unified way under mild conditions with no need for resorting to discretization arguments or exponential probability estimates. Instead, one uses ordinary tightness and weak convergence for a family of controlled versions of the original processes.  

The rest of this paper is organized as follows. Section~\ref{SectLDP} is dedicated to the statement and proof of the large deviation principle under general hypotheses. In Section~\ref{SectLipschitz}, we verify the hypotheses for coefficients that are locally Lipschitz continuous with sublinear growth at infinity but may depend on the past as well as the large deviation parameter $\epsilon$. This result yields, as an application, the large deviation principle obtained in \citet{mohammedzhang} for a class of systems with memory or delay; see Section~\ref{SectApplDelay}. The approach used here actually allows to easily handle more general delay models than the point delay studied in \citet{mohammedzhang}; cf.\ Remark~\ref{RemDelayModels} below. In Section~\ref{SectApplHoelder}, we derive the large deviation principle for a class of positive It\^o diffusions with dispersion coefficient of square-root type such as, for example, the Cox-Ingersoll-Ross (\mbox{CIR}) process, which serves as a model for interest rates in mathematical finance. The result is essentially the same as the large deviation principle for positive diffusions obtained in \citet{baldicaramellino}; it might be compared to Theorem~1.3 in \citet{donatimartinetal04}, which also covers degenerate cases (zero initial condition or drift vanishing in zero). The appendix contains the variational formula for Laplace functionals of Brownian motion obtained in \citet{bouedupuis} as well as two related technical results.

%-------

\section{General large deviation principle}
\label{SectLDP}

Let $d,m\in \N$, and let $T>0$. For $n\in\N$, set $\mathcal{W}^n\doteq\mathbf{C}([0,T],\R^n)$ and endow $\mathcal{W}^n$ with the standard topology of uniform convergence. For $\epsilon>0$, let $b_\epsilon$ and $b$ be functions mapping $[0,T]\times\mathcal{W}^d$ to $\R^d$, and $\s_\epsilon$ and $\s$ functions mapping $[0,T]\times\mathcal{W}^d$ to $\R^{d\times m}$. Let $(\mathcal{W}^m,\mathcal{B},\theta)$ be the canonical probability space with Wiener measure $\theta$, and let $W$ be the coordinate process. Thus $W$ is an $m$-dimensional standard Brownian motion with respect to $\theta$. Let $(\mathcal{G}_t)$ be the $\theta$-augmented filtration generated by $W$, and let $\mathcal{M}^2[0,T]$ denote the space of $\R^{m}$-valued square-integrable $(\mathcal{G}_t)$-predictable processes.

Fix $x\in \R^d$. For $\epsilon>0$, we consider the It{\^o} stochastic differential equation
\begin{equation}\label{SDE2}
      dX^\epsilon_t = b_\epsilon(t,X^\epsilon)\,dt + \sqrt{\epsilon}\,\s_\epsilon(t,X^\epsilon)\,dW_t,
\end{equation}
and, with $v\in \mathcal{M}^2[0,T]$, its controlled counterpart
\begin{equation}\label{ControlSDE2}
      dX^{\epsilon,v}_t=b_\epsilon(t,X^{\epsilon,v})\,dt + \s_\epsilon(t,X^{\epsilon,v})v_t\,dt+\sqrt{\epsilon}\,\s_\epsilon(t,X^{\epsilon,v})\,dW_t,
\end{equation}
both over the time interval $[0,T]$ and with initial condition $X^{\epsilon,v}_0 = X^\epsilon_0 = x$.
Observe that if $\epsilon = 0$, then Equation~\eqref{SDE2} becomes a deterministic functional equation, namely
\begin{equation}
	\phi_t=x+\int_0^t b(s,\phi)\,ds.
\end{equation}
Similarly, if $\epsilon = 0$ and we pick $v = f\in L^2([0,T];\R^m)$, then Equation~\eqref{ControlSDE2} reduces to
\begin{equation}\label{IntegralEquation}
	\phi_t= x + \int_0^t b(s,\phi)\,ds + \int_0^t \s(s,\phi)f_s\,ds.
\end{equation}

Let us introduce the following assumptions:
\begin{description}
	\item[H1] The coefficients $b$ and $\s$ are predictable. Moreover, $b(t,\bcdot)$, $\s(t,\bcdot)$ are uniformly continuous on compact subsets of $\mathcal{W}^d$, uniformly in $t\in[0,T]$, and $t\mapsto \s(t,\phi)$ is in $L^2([0,T];\R^d)$ for any $\phi\in\mathcal{W}^d$.

	\item[H2] The coefficients $b_\epsilon$, $\s_\epsilon$ are predictable maps such that $b_\epsilon\to b$ and $\s_\epsilon\to \s$ as $\epsilon \to 0$ uniformly on $[0,T]\times \mathcal{W}^d$.

	\item[H3] For all $\epsilon > 0$ sufficiently small, pathwise uniqueness and existence in the strong sense hold for Equation~\eqref{SDE2}.
  
	\item[H4] For any $f\in L^2([0,T];\R^m)$, Equation~\eqref{IntegralEquation} has a unique solution so that the map
  \[
  \Gamma_x:L^2([0,T];\R^m)\longrightarrow \mathcal{W}^d
  \]
  which takes $f\in L^2[0,T]$ to the solution of Equation~\eqref{IntegralEquation} is well defined.
  
	\item[H5] For all $N\in\N$, the map $\Gamma_x$ is continuous when restricted to
  \[
  S_N\doteq \left\{f\in L^2([0,T],\R^m): \int_0^T |f_s|^2 ds\leq N\right\}
  \]
  endowed with the weak topology of $L^2[0,T]$.

	\item[H6] If $\{\epsilon_n\} \subset (0,1]$ is such that $\epsilon_n\to 0$ as $n\to\infty$ and $\{v_n\}_{n\in\N}\subset\mathcal{M}^2[0,T]$ is such that, for some constant $N>0$,
  \[
  	\sup_{n\in\N}\int_0^T |v^n_s(\omega)|^2\,ds \leq N \quad\text{for $\theta$-almost all }\omega\in\mathcal{W}^m,
  \]
  then $\{X^{\epsilon_n,v_n}\}_{n\in\N}$ is tight as a family of $\mathcal{W}^d$-valued random variables and
  \[
  \sup_{n\in\N} \int_0^T \E\left[|\s(s,X^{\epsilon_n,v_n})|^2\right]ds < \infty.
  \]
\end{description}

\begin{remark} \label{RemH2}
We shall see in Section~\ref{SectLipschitz} that assumption H2 can be weakened. Specifically, we shall require uniform convergence of $b_\epsilon$, $\s_\epsilon$ to $b$ and $\s$, respectively, only on bounded subsets of $\mathcal{W}^d$.
\end{remark}

\begin{remark} \label{RemH4}
As will be clear from the proof of Theorem~\ref{ThLDP}, existence of solutions to Equation~\eqref{IntegralEquation} is a consequence of hypotheses H1--H3 and H6. Thus hypothesis H4 reduces to the requirement of uniqueness of solutions for the deterministic integral equation \eqref{IntegralEquation}.
\end{remark}

\begin{remark} The spaces $S_N\doteq \{f\in L^2([0,T],\R^m): \int_0^T |f_s|^2 ds\leq N\}$, $N\in \N$, introduced in Hypothesis H5 are compact Polish spaces when endowed with the weak topology of $L^2[0,T]$. Continuity of the restriction of $\Gamma_x$ to $S_{N}$ as required by H5 only is needed to guarantee that the rate function has compact sublevel sets, and accordingly is good.
\end{remark}

\begin{theorem} \label{ThLDP}
Grant H1--H6. Then the family $\{X^\epsilon\}_{\epsilon>0}$ of solutions of the stochastic differential equation \eqref{SDE2} with initial condition $X^\epsilon_0=x$ satisfies the Laplace principle with good rate function $I_x: \mathcal{W}^{d} \rightarrow [0,\infty]$ given by
\[
	I_x(\phi)=\inf_{\{f\in L^2([0,T];\R^m): \Gamma_x(f)=\phi\}}\frac{1}{2}\int_0^T |f_t|^2\,dt
\]
whenever $\{f\in L^2([0,T];\R^m): \Gamma_x(f)=\phi\}\neq \emptyset$, and $I_x(\phi)=\infty$ otherwise.
\end{theorem}

\begin{proof}[Proof of the lower bound.] The first step in proving Theorem~\ref{ThLDP} is the Laplace principle lower bound. We have to show that for any bounded and continuous function $F\!: \mathcal{W}^d \rightarrow \R$,
\begin{equation} \label{eq:lowerbound}
	\liminf_{\epsilon\to0+}-\epsilon \log\E\biggl[e^{-\frac{F(X^\epsilon)}{\epsilon}}\biggl]\geq\inf_{\phi\in\mathcal{W}^d}\{F(\phi)+I_x(\phi)\}.
\end{equation}
It suffices to prove that any sequence $\{\epsilon_n\}_{n\in\N} \subset (0,1]$ such that $\epsilon_n\to0$ as $n\to\infty$ has a subsequence for which the above limit relation holds.

Let $\{\epsilon_n\}_{n\in\N} \subset (0,1]$ be such that $\epsilon_n\to0$. By assumption H3, for any $n\in\N$, $X^n \doteq X^{\epsilon_n}$ is a strong solution of Equation \eqref{SDE2}. Hence there exists a measurable map $h^n\!:\mathcal{W}^m\to\mathcal{W}^d$ such that $X^n=h^n(W)$ $\theta$-almost surely. Representation formula \eqref{RepresentationFormulaC} in the appendix applies and yields
\begin{multline} \label{eq:representation}
	-\epsilon_n\log\E\left[e^{-\frac{F(X^n)}{\epsilon_n}}\right] = -\epsilon_n\log\E\left[e^{-\frac{F\circ h^n(W)}{\epsilon_n}}\right] \\
	=\epsilon_n \inf_{v\in \mathcal{M}^2[0,T]}\E\left[\frac{1}{2}\int_0^T|v_s|^2\,ds+\frac{1}{\epsilon_n}F\circ h^n\left(W+\int_{0}^{\bcdot} v_s\,ds\right)\right] \\
	=\inf_{v\in \mathcal{M}^2[0,T]}\E\left[\frac{1}{2}\int_0^T|v_s|^2\,ds+F\circ h^n\left(W+\frac{1}{\sqrt{\epsilon_n}}\int_{0}^{\bcdot} v_s\,ds\right)\right].
\end{multline}

Fix $\delta>0$. We claim that there exists a constant $N > 0$ such that for every $n\in\N$ there exists $v^n\in \mathcal{M}^2[0,T]$ such that $\int_0^T|v_s^n|^2 ds\leq N$ and
\begin{multline}\label{eq:lowerrepre}
	-\epsilon_n\log\E\left[e^{-\frac{F(X^n)}{\epsilon_n}}\right]\\
	\geq \E\left[\frac{1}{2}\int_0^T|v^n_s|^2\,ds + F\circ h^n\left(W+\frac{1}{\sqrt{\epsilon_n}}\int_{0}^{\bcdot} v^n_s\,ds\right)\right] - \delta.
\end{multline}
Indeed, by the definition of infimum, for any $n\in\N$ there exists $u^n\in \mathcal{M}^2[0,T]$ such that
\begin{multline*}
	-\epsilon_n\log\E\left[e^{-\frac{F(X^n)}{\epsilon_n}}\right] \\
	\geq\E\left[\frac{1}{2}\int_0^T|u^n_s|^2\,ds + F\circ h^n\left(W+\frac{1}{\sqrt{\epsilon_n}}\int_{0}^{\bcdot} u^n_s\,ds\right) \right]-\frac{\delta}{2}.
\end{multline*}
Setting $M\doteq \|F\|_\infty$, it follows that
\begin{equation} \label{eq:tightness}
	\sup_{n\in\N} \E\left[\frac{1}{2}\int_0^T|u^n_s|^2\,ds\right] \leq 2M + \frac{\delta}{2} < \infty.
\end{equation}
For $N\in \mathbb{N}$ define the stopping time
\[
	\tau^n_N\doteq \inf\left\{t\in[0,T]: \int_0^t|u^n_s|^2\,ds\geq N\right\}\wedge T.
\]
The processes $u^{n,N}_s \doteq u^n_s\cdot \boldsymbol{1}_{[0,\tau^n_N]}(s)$ belong to $\mathcal{M}^2[0,T]$ with $\int_0^T|u^{n,N}_s|^2\,ds\leq N$. By Chebychev's inequality and \eqref{eq:tightness},
\[
	\theta\left(u^n\neq u^{n,N}\right)\leq \theta\left(\int_0^T|u^n_s|^2\,ds\geq N\right)\leq \frac{4M+\delta}{N}.
\]
This observation implies that
\begin{multline}\label{eq:approx}
	-\epsilon_n\log\E\left[e^{-\frac{F(X^n)}{\epsilon_n}}\right] \\
	\geq \E\left[\frac{1}{2}\int_0^T|u^{n,N}_s|^2\,ds + F\circ h^n\left(W+\frac{1}{\sqrt{\epsilon_n}}\int_0^{\bcdot}u^{n,N}_s\,ds\right)\right] \\
	- \frac{2M(4M+\delta)}{N}-\frac{\delta}{2}.
\end{multline}
In view of \eqref{eq:approx}, to verify the claim, we take $N$ big enough so that
\[
	\frac{2M(4M+\delta)}{N}<\frac{\delta}{2}
\]
and set, for $n\in \mathbb{N}$, $v^n\doteq u^{n,N}$.

Choose $N$ and $\{v^n\}\subset\mathcal{M}^2[0,T]$ according to the claim, $\delta > 0$ being fixed. Thanks to hypothesis H3 and Lemma~\ref{LemmaControlSDERep} in the appendix, the controlled stochastic equation
\[
	dX^{n,v^n}_t=b_{\epsilon_n}(t,X^{n,v^n})\,dt+\s_{\epsilon_n}(t,X^{n,v^n})v^n_t\,dt+\sqrt{\epsilon_n}\s_{\epsilon_n}(t,X^{n,v^n})\,dW_t
\]
possesses a unique strong solution with $X^{n,v^n}_0 = x$, and
\begin{equation}\label{EqControl}
	h^n\left(W+\frac{1}{\sqrt{\epsilon_n}}\int_0^{\bcdot}v^{n}_s\,ds\right) = X^{n,v^n}\qquad \theta\text{-a.e}.
\end{equation}
It follows that, for any $n\in\N$, we can rewrite \eqref{eq:lowerrepre} to obtain
\[
-\epsilon_n \log\E\left[e^{-\frac{F(X^n)}{\epsilon_n}}\right]\geq \E\left[\frac{1}{2}\int_0^T|v^{n}_s|^2\,ds + F(X^{n,v^n})\right]-\delta,
\]
where $X^{n,v^n}$ is the unique strong solution of Equation~\eqref{ControlSDE2} with $\epsilon = \epsilon_n$ and control $v = v_{n}$.

Next we check that $\{(X^{n,v^n},v^n)\}_{n\in\mathbb{N}}$ is tight as a family of random variables with values in $\mathcal{W}^d\times S_N$. Since both $S_N$ and $\mathcal{W}^d$ are Polish spaces, it suffices to show that $\{X^{n,v^n}\}_{n\in\N}$ is tight as a family of $\mathcal{W}^d$-valued random variables and $\{v^n\}_{n\in\N}$ is tight as a family of $S_N$-valued random variables. But tightness of $\{X^{n,v^n}\}_{n\in\N}$ follows by assumption H6, while tightness of $\{v^n\}$ is automatic since $S_{N}$ is compact. Therefore, possibly taking a subsequence, we have that $(X^{n,v^n},v^n)$ converges in distribution to a $\mathcal{W}^d\times S_N$-valued random variable $(X,v)$ defined on some probability space $(\Omega,\F,\PR)$. Let us denote by $\E_{\PR}$ expectation with respect to the measure $\PR$. We are going to show that $X$ satisfies
\begin{equation}\label{limitEq}
	X_t = x+\int_0^t b(s,X)\,ds+\int_0^t\s(s,X)v_s\,ds\qquad\PR\text{-a.s.}
\end{equation}
To this end, for $t\in [0,T]$, consider the map $\Psi_t\!: \mathcal{W}^d\times S_N\to\R$ defined by
\[
	\Psi_t(\phi,f)\doteq \left|\phi(t)-x-\int_0^t b(s,\phi(s))\,ds-\int_0^t\s(s,\phi(s))f_s\,ds\right|\wedge 1.
\]
Clearly, $\Psi_t$ is bounded. Moreover, $\Psi_t$ is continuous. Indeed, let $\phi^n\to\phi$ in $\mathcal{W}^d$ and $f^{n}\to f$ in $S_N$ with respect to the weak topology of $L^2$. The set $\mathcal{C}\doteq \{\phi^n:n\in\N\}\cup\{\phi\}$ is a compact subset of $\mathcal{W}^d$. Therefore, by assumption H1, there exist moduli of continuity $\rho_b$ and $\rho_\s$ mapping $[0,\infty[$ into $[0,\infty[$ such that $|b(s,\phi)-b(s,\psi)|\leq \rho_b(\|\phi-\psi\|_\infty)$ and $|\s(s,\phi)-\s(s,\psi)|\leq \rho_\s(\|\phi-\psi\|_\infty)$ for all $s\in[0,T]$ and all $\phi,\psi\in \mathcal{C}$. Using H\"older's inequality and the fact that $\|f^n\|_{L^2}\leq \sqrt{N}$, we find
\begin{align*}
	&|\Psi_t(\phi^n,f^n)-\Psi_t(\phi,f)| \\
	&\leq |\phi^n_t-\phi_t| + \int_0^t |b(s,\phi^n)-b(s,\phi)|\,ds \\
	&+\quad \int_0^t|\s(s,\phi)-\s(s,\phi^n)|\cdot|f^n_s|\,ds + \left|\int_0^t \s(s,\phi)\left(f - f^n_s\right)ds \right| \\
	&\leq\|\phi^n-\phi\|_\infty + T\cdot \rho_b(\|\phi^n-\phi\|_\infty) + \sqrt{N\cdot T}\cdot\rho_\s(\|\phi^n-\phi\|_\infty) \\
	&\qquad +\left|\int_0^t\s(s,\phi)\left(f_s- f^n_s\right)ds\right|.
\end{align*}
The terms involving $\|\phi-\phi^{n}\|_{\infty}$ in the above display go to zero as $n\to \infty$. Thanks to hypothesis H1, the function $\s(\cdot,\phi)$ is in $L^2[0,T]$; since $f^n$ converges weakly to $f$, the rightmost term of the previous display goes to zero as well. This shows that $\Psi_t$ is continuous. Since $(X^{n,v^n},v^n)$ converges in distribution to $(X,v)$ and $\Psi_t$ is bounded and continuous, the continuous mapping theorem for weak convergence implies that
\begin{equation}\label{eq:limitForX}
    \lim_{n\to\infty}\E\left[\Psi_t(X^{n,v^n},v^n)\right] = \E_{\PR}\left[\Psi_t(X,v)\right].
\end{equation}
If we show that the limit in \eqref{eq:limitForX} is actually zero, then, by definition of $\Psi_t$, $X$ will satisfy Equation~\eqref{limitEq} $\PR$-almost surely for all $t\in[0,T]$. Since $X$ has continuous paths, it follows that $X$ satisfies Equation~\eqref{limitEq} for all $t\in[0,T]$, $\PR$-almost surely. Observe that
\begin{multline*}
	\E[\Psi_t(X^{n,v^n},v^n)]\leq \E\left[\int_0^t|b_{\epsilon_n}(s,X^{n,v^n})-b(s,X^{n,v^n})|\,ds\right]\\
	+\E\left[\int_0^t|\s_{\epsilon_n}(s,X^{n,v^n})-\s(s,X^{n,v^n})|\cdot|v_s^{n}|\,ds\right] \\
	+\sqrt{\epsilon_n}\E\left[\Bigl|\int_0^t \s_{\epsilon_n}(s,X^{n,v^n})\,dW_s\Bigr|\right].
\end{multline*}
Using the uniform convergence of $\s_\epsilon$ to $\s$ and of $b_\epsilon$ to $b$ on $[0,T]\times \mathcal{W}^d$ according to H2, we get
\begin{multline*}
	\E[\Psi_t(X^{n,v^n},v^{n})]\leq t\|b_{\epsilon_n}-b\|_\infty \\
	+\|\s_{\epsilon_n}-\s\|_\infty \E\left[\int_0^T|v_s^{n}|\,ds\right] + \sqrt{\epsilon_n}\sqrt{\int_0^t\E\left[ |\s_{\epsilon_n}(s,X^{n,v^n})|^2\right]\,ds},
\end{multline*}
which goes to zero as $n\to\infty$. The last term in the above display tends to zero since
\begin{multline*}
	\sup_{n\in\N} \int_0^t\E\left[ |\s_{\epsilon_n}(s,X^{n,v^n})|^2\right]\,ds
	\leq 2\sup_{n\in\N} \int_0^T\E\left[ |\s(s,X^{n,v^n})|^2\right]\,ds \\
	+ 2\sup_{n\in\N} \int_0^T \E\left[ |\s_{\epsilon_n}(s,X^{n,v^n})-\s(s,X^{n,v^n})|^2\right]\,ds\\
\leq 2T\sup_{n\in\N}\|\s_{\epsilon_n}-\s\|^2_\infty + 2\sup_{n\in\N} \int_0^T\E\left[ |\s(s,X_.^{n,v^n})|^2\right]\,ds,
\end{multline*}
which is finite thanks to hypothesis H6 (and H2). Recalling \eqref{eq:limitForX}, we have shown that
\[
	\lim_{n\to\infty}\E[\Psi_t(X^{n,v^n},v^{n})] = \E_{\PR}[\Psi_t(X,v)]=0.
\]
Thus $X$ satisfies Equation~\eqref{limitEq} for all $t\in[0,T]$, $\PR$-almost surely. If $f\in L^{2}([0,T];\R^m)$, then applying the same argument to the (constant) sequence of deterministic control processes $v^{n} = f$, one finds that Equation~\eqref{IntegralEquation} possesses a solution. The existence part of hypothesis H4 is therefore a consequence of hypotheses H1, H2, H3, and H6. 

The mapping $S_N\ni f\to \int_0^T |f_s|^2 ds\in \R$ is nonnegative and lower semicontinuous (with respect to the weak $L^2$-topology on $S_{N}$). Since the trajectories of $v^n$ are in $S_{N}$ for all $n\in \N$ and $v^n$ converges in distribution to $v$, a version of Fatou's lemma (Theorem A.3.12 in \citet[p.\,307]{dupuisellis}) entails that
\[
	\liminf_{n\to \infty} \E\left[\int_0^T |v^n_s|^2\,ds\right]\geq \E_{\PR}\left[\int_0^T|v_s|^2\,ds\right].
\]
Using this inequality and the continuous mapping theorem (recalling that $F$ is bounded and continuous) we find that
\begin{align*}
	\liminf_{n\to\infty}& -\epsilon_n\log\E\left[e^{-\frac{F(X^n)}{\epsilon_n}}\right]\\
	&\geq \liminf_{n\to\infty} \E\left[\frac{1}{2}\int_0^T|v^n_s|^2\,ds + F(X^{n,v^n})\right] - \delta\\
%	&\geq \liminf_{n\to\infty} \E\left[\frac{1}{2}\int_0^T|v^n_s|^2\,ds\right] + \lim_{n\to\infty}\E\left[F(X^{n,v^n})\right] - \delta\\
	&\geq \E_{\PR}\left[\frac{1}{2}\int_0^T|v_s|^2\,ds + F(X^v)\right]-\delta\\
	&\geq \inf_{\{(f,\phi)\in L^2\times\mathcal{W}^d: \phi=\Gamma_x(f)\}}\left\{ \frac{1}{2}\int_0^T |f_s|^2\,ds + F(\phi)\right\} -\delta\\
	&\geq\inf_{\phi\in\mathcal{W}^d}\{I_x(\phi)+F(\phi)\}-\delta.
\end{align*}
The second but last inequality is obtained by evaluating the random variable inside the expectation $\omega$ by $\omega$. Since $\delta$ was arbitrary, the lower bound follows.
\end{proof}

\begin{proof}[Proof of the upper bound.]  We now prove the Laplace principle upper bound,
\begin{equation} \label{eq:upperbound}
	\limsup_{\epsilon\to 0} -\epsilon\log\E\left[e^{-\frac{F(X^\epsilon)}{\epsilon}}\right]\leq\inf_{\phi\in\mathcal{W}^{d}}\{I_x(\phi)+ F(\phi)\}
\end{equation}
for $F\!: \mathcal{W}^d \rightarrow \R$ bounded and continuous. As for the lower bound, it suffices to show that any sequence $\{\epsilon_n\}_{n\in\N} \subset (0,1]$ such that $\epsilon_n\to0$ has a subsequence for which the limit in \eqref{eq:upperbound} holds.

Fix $\delta>0$. If the infimum in \eqref{eq:upperbound} is not finite, the inequality is trivially satisfied; hence we may assume that the infimum is finite. Since $F$ is bounded, there exists $\phi\in\mathcal{W}^d$ such that
\begin{equation}\label{eq:remarkphi}
    I_x(\phi)+F(\phi)\leq \inf_{\psi\in\mathcal{W}^{d}}\{I_x(\psi) + F(\psi)\}+\frac{\delta}{2}<\infty.
\end{equation}
For such $\phi$, choose $\tilde{v}\in L^2([0,T];\R^m)$ such that
\[
	\frac{1}{2}\int_0^T|\tilde{v}_s|^2\,ds\leq I_x(\phi) + \frac{\delta}{2},
\]
and $\phi = \Gamma_x(\tilde{v})$. This choice is possible by the definition of $I_x$ and since $I_x(\phi)<\infty$. Let $\{\epsilon_n\}_{n\in\N} \subset (0,1]$ be such that $\epsilon_n\to0$ as $n\to\infty$. For $n\in\N$, let $X^{n,\tilde{v}}$ be the unique strong solution of Equation~\eqref{ControlSDE2} with $\epsilon = \epsilon_n$ and (deterministic) control $v = \tilde{v}$. Then the family $\{(X^{n,\tilde{v}},\tilde{v})\}_{n\in\N}$ is tight. Therefore, possibly taking a subsequence, $(X^{n,\tilde{v}},\tilde{v})$ converges in distribution to a random variable $(X,\tilde{v})$ defined on some probability space $(\Omega,\F,\PR)$. As in the proof of the lower bound, it follows that, $\PR$-almost surely,
\[
	X_t = x+\int_0^t b(s,X)\,ds+\int_0^t\s(s,X)\tilde{v}_s\,ds\quad\text{for all }t\in[0,T].
\]
The above integral equation, which is deterministic since $\tilde{v}\in L^2([0,T];\R^m)$ is deterministic, coincides with Equation~\eqref{IntegralEquation}. The solution to that equation is unique by assumption H4, hence $X = \Gamma_{x}(\tilde{v}) = \phi$ $\PR$-almost surely. Using representation \eqref{eq:representation}, we obtain
\begin{align*}
	\limsup_{n\to\infty}& -\epsilon_n\log\E\left[e^{-\frac{F(X^{\epsilon_n})}{\epsilon_n}}\right]\\
	&=\limsup_{n\to\infty} \inf_{v\in \mathcal{M}^2[0,T]}\E\left[\frac{1}{2}\int_0^T |v_s|^2\,ds+F\circ h^n\left(W+\frac{1}{\sqrt{\epsilon_n}}\int_0^{\bcdot}v_s\,ds\right)\right]\\
	&\leq \limsup_{n\to\infty} \E\left[\frac{1}{2}\int_0^T |\tilde{v}_s|^2\,ds+F(X^{n,\tilde{v}})\right]\\
	&=\frac{1}{2}\int_0^T |\tilde{v}_s|^2\,ds + \lim_{n\to\infty} \E\left[F(X^{n,\tilde{v}})\right]\\
	&\leq I_x(\phi) + \frac{\delta}{2} + \lim_{n\to\infty} \E\left[F(X^{n,\tilde{v}})\right].
\end{align*}
Since $F$ is bounded and continuous and $X^{n,\tilde{v}}$ converges in distribution to $X=\phi$, we have $\lim_{n\to\infty} \E[F(X^{n,\tilde{v}})]= F(\phi)$. Thanks to \eqref{eq:remarkphi}, we can end the above chain of inequalities by
\[
	I_x(\phi)+\frac{\delta}{2}+F(\phi)\leq \inf_{\psi\in\mathcal{W}^d} \left\{I_x(\psi)+ F(\psi)\right\} + \delta.
\]
Since $\delta > 0$ is arbitrary, the proof of the Laplace principle upper bound is complete.
\end{proof}

\begin{proof}[Goodness of the rate function.] To prove that $I_x$ is actually a good rate function, it remains to check that $I_x$ has compact sublevel sets. This follows from the compactness of $S_N$ for any $N >0$, and by the continuity on these sets of the map $\Gamma_x$, which takes $v$ to the unique solution of Equation~\eqref{IntegralEquation}, according to assumption H5. Indeed $\{\phi\in\mathcal{W}^d: I_x(\phi)\leq N\} = \bigcap_{\epsilon>0} \Gamma_x(S_{N+\epsilon})$ is the intersection of compact sets, hence compact.
\end{proof}

%-------

\section{Locally Lipschitz continuous coefficients}
\label{SectLipschitz}

In this section we show that hypotheses H1--H6 hold in the important case of locally Lipschitz continuous coefficients which are predictable and satisfy a sublinear growth condition. With the notation of Section~\ref{SectLDP}, let us introduce the following assumptions:
\begin{description}
	\item[A1] $b$ and $\s$ satisfy a sublinear growth condition. Specifically, there exists $M>0$ such that for all $t\in[0,T]$, all $\phi\in\mathcal{W}^d$,
\[
	|b(t,\phi)| \vee |\s(t,\phi)| \leq M\left(1+\sup_{s\in [0,t]}|\phi_s|\right).
\]
	\item[A2] $b$ and $\s$ are locally Lipschitz continuous. Specifically, for any $R>0$ there exists $L_R>0$ such that for all $t\in[0,T]$, all $\phi, \tilde{\phi}\in\mathcal{W}^d$ with $\sup_{s\in[0,t]}|\phi_s|\vee|\tilde{\phi}_s|\leq R$,
\[
	|b(t,\phi)-b(t,\tilde{\phi})| \vee |\s(t,\phi)-\sigma(t,\tilde{\phi})| \leq L_R\sup_{s\in[0,t]}|\phi_s-\tilde{\phi}_s|.
\]

\item[A3] $b_\epsilon$ and $\s_\epsilon$ enjoy property A1 (with the same constant $M$ as $b$, $\s$) as well as property A2.

\item[A4] $b_\epsilon,\s_\epsilon$ converge as $\epsilon\to0$ to $b$ and $\s$, respectively, uniformly on bounded subsets of $[0,T]\times \mathcal{W}^d$.
\end{description}

\begin{remark} We distinguish between hypotheses A1--A2 and A3 since A3 is not needed to verify H4 and H5. Observe that A4 is not exactly H2, indeed the convergence is not on the whole $\mathcal{W}^d$, but on the bounded subsets of $\mathcal{W}^d$.
\end{remark}

\begin{remark} \label{RemA2}
Assumption A2 implies that if $0\leq t\leq T$ and $\phi,\psi\in\mathcal{W}^d$ are such that $\phi_s=\psi_s$ for all $s\in[0,t]$, then $b(t,\phi)=b(t,\psi)$ and the process $\{b(t,\cdot)\}_{t\geq 0}$ is adapted to the canonical filtration. In particular, if $\{b(t,\cdot)\}_{t\geq 0}$ is c\`adl\`ag, then $b$ is also predictable. The same remark is also true for $\s$.
\end{remark}

\begin{theorem} \label{ThLDPLip}
Grant A1--A4. Then the family $\{X^\epsilon\}_{\epsilon>0}$ of solutions of the stochastic differential equation \eqref{SDE2} with initial condition $X^\epsilon_0 = x$ satisfies the Laplace principle with good rate function $I_x: \mathcal{W}^{d} \rightarrow [0,\infty]$ given by
\[
	I_x(\phi)=\inf_{\{f\in L^2([0,T];\R^m):\phi = \Gamma_x(f)\}}\frac{1}{2}\int_0^T |f_t|^2\,dt
\]
whenever $\{f\in L^2([0,T];\R^m):\phi = \Gamma_x(f)\}\neq \emptyset$, and $I_x(\phi)=\infty$ otherwise.
\end{theorem}

To prove Theorem~\ref{ThLDPLip} it is enough to show that hypotheses H1--H6 of Theorem~\ref{ThLDP} are entailed by assumptions A1--A4. As mentioned above, we will not be able to prove H2. Instead, we are going to show that in this special setting H2 is not really needed; this discussion is postponed to the end of the section.

\begin{proof}[Hypotheses H1, H3.] H1 is satisfied, in fact $b(t,\cdot)$ and  $\s(t,\cdot)$ are uniformly continuous on bounded subsets of $\mathcal{W}^d$, uniformly in $t\in[0,T]$ because of assumption A2. Moreover $\s(\cdot,\phi)$ belongs to $L^2[0,T]$ for any $\phi\in\mathcal{W}^d$ since
\[
	\sup_{t\in[0,T]}|\s(t,\phi)|^2\leq 2M^2(1+\|\phi\|_\infty^2)
\]
as a consequence of A1. Assumption A3 implies that pathwise uniqueness and existence of strong solutions hold for Equation~\eqref{SDE2}; see, for instance, Theorem 12.1 in \citet[p.\,132]{rogerswilliams}.  
\end{proof}

\begin{proof}[Hypotheses H4, H5.]
In view of Remark~\ref{RemH4}, to verify H4 it suffices to show that, given any $f\in L^{2}([0,T],\R^m)$, there is a unique solution $\phi\in \mathbf{C}([0,T],\R^d)$ of Equation~\eqref{IntegralEquation}. Moreover, for $\phi$ we have the growth estimate
\begin{equation} \label{IntEqBound}
	\sup_{0\leq s\leq t} |\phi_t|^2\leq \left(3|x|^2 + 6M^2t^2 + 6M^2t\|f\|^2\right) e^{6M^2t(t+\|f\|^2)},\quad t\in[0,T].
\end{equation}
To verify that uniqueness holds, let $\phi, \psi\in \mathbf{C}([0,T],\R^d)$ be solutions of Equation~\eqref{IntegralEquation}. Then for $t\in [0,T]$,
\[
	|\phi_t-\psi_t|\leq \int_0^t |b(s,\phi)-b(s,\psi)|\,ds + \int_0^t |\s(s,\phi)-\s(s,\psi)|\cdot |f_s|\,ds.
\]
By taking the square, using H\"older's inequality and the local Lipschitz continuity according to A2, we obtain for $R>0$ big enough (since $\phi$, $\psi$ are bounded),
\[
	|\phi_t-\psi_t|^2\leq 2L^2_R\left(T+\|f\|^2\right)\cdot \int_0^t \sup_{u\in[0,s]} |\phi_u-\psi_u|^2\,ds.
\]
Gronwall's inequality now entails that $\|\phi-\psi\|_\infty = 0$, which yields uniqueness. Similarly, also using the sublinear growth condition A1, one finds that
\begin{align*}
	|\phi_t|^2&\leq 3|x|^2 + 3t\int_{0}^t |b(s,\phi)|^2\,ds + 3\left(\int_{0}^t |\s(s,\phi)|\cdot |f_s|\,ds\right)^2 \\
	&\leq 3|x|^2 + 6M^2\left(t+\|f\|^2\right) \int_{0}^t \left(1+\sup_{0\leq u\leq s}|\phi_u|^2 \right)ds\\
	&\leq 3|x|^2 + 6M^2t^2 + 6M^2t\|f\|^2 + 6M^2\left(t+\|f\|^2\right) \int_{0}^t\sup_{0\leq u\leq s}|\phi_u|^2\,ds.
\end{align*}
An application of Gronwall's inequality now yields the growth estimate \eqref{IntEqBound}.

In order to establish H5, we have to show that, given any $N\in \N$, the map $\Gamma_x$ defined in $H4$ is continuous when restricted to $S_{N}$. Recall that $S_N$ is a compact Polish space. Take $\{f^n\}\subset S_N$ such that $f^n\to f$ weakly, and define $\phi^n\doteq \Gamma_x(f^n)$, $\phi\doteq \Gamma_x(f)$. Then for $t\in [0,T]$,
\begin{multline*}
	\phi^n_t-\phi_t= \int_0^t \left(b(s,\phi^n)-b(s,\phi)\right)ds\\
	+ \int_0^t \left(\s(s,\phi^n)-\s(s,\phi)\right)f^n_s\,ds + \int_0^t \s(s,\phi)\left(f^n_s-f_s\right)ds.
\end{multline*}
Since $\|f^n\|^2\leq N$, estimate \eqref{IntEqBound} yields that $R\doteq \sup_{n\in\N}\|\phi^n\|_\infty \vee \|\phi\|_{\infty}$ is finite. Therefore, using A2,
\begin{multline*}
	\sup_{u\in[0,t]}|\phi^n_u-\phi_u|\leq L_{R} \int_0^t \sup_{u\in[0,s]} |\phi^n_u-\phi_u|\,ds\\
	+ L_{R} \int_0^t \sup_{u\in[0,s]}|\phi^n_u-\phi_u|\cdot|f^n_s|\,ds + \sup_{u\in[0,T]}\left|\int_0^u \s(s,\phi)\left(f^n_s-f_s\right)ds\right|.
\end{multline*}
Set $\Delta_\s^n\doteq \sup_{u\in[0,T]} \left|\int_0^u \s(s,\phi)\left(f^n_s-f_s\right)ds\right|$. By H\"older's inequality and since $\|f^n\|^2\leq N$ for all $n\in\N$, it follows that
\[
	\sup_{u\in[0,t]}|\phi^n_u-\phi_u|^2 \leq 3L_{R}^2(t+N) \int_0^t \sup_{u\in[0,s]} |\phi^n_u-\phi_u|^2\,ds + 3(\Delta_\s^n)^2.
\]
An application of Gronwall's lemma yields
\[
	\|\Gamma_x(f^n)-\Gamma_x(f)\|_{\infty} = \sup_{t\in[0,T]}|\phi^n_t-\phi_t|^2 \leq 3(\Delta_\s^n)^2 e^{3L^2T(T+N)}.
\]
In order to establish continuity of $\Gamma_{x}$ on $S_{N}$, it remains to check that $\Delta_\s^n$ goes to zero as $n\to \infty$. Thanks to assumption A1, the function $\s(\bcdot,\phi)$ is in $L^\infty[0,T]$. It follows that $\s(\bcdot,\phi)f^n$ converges weakly to $\s(\bcdot,\phi)f$ in $L^2$. Moreover, the family $\{\s(\bcdot,\phi)f^n\}_{n\in\N}$ is bounded in $L^2$ with respect to the $L^2$-norm. Hence
\[
	\int_0^t \s(s,\phi)f^n_s\,ds\stackrel{n\to\infty}{\longrightarrow} \int_0^t \s(s,\phi)f_s\,ds,
\]
uniformly in $t\in[0,T]$, which implies $\Delta_\s^n \to 0$ as $n\to\infty$.
%In fact, consider a bounded sequence $g^n\in L^2([0,T],\R^d)$ which converges weakly to $g\in L^2$. Define
%\begin{align*}
%	F_n(t)\doteq \int_0^t g^n_s\,ds,& &F(t)\doteq \int_0^t g_s\,ds,& &t\in [0,T].&
%\end{align*}
%Then for $0\leq s\leq t\leq T$,
%\[
%|F_n(t)-F_n(s)|\leq \int_s^t |g^n_r|\,dr\leq  \sqrt{t-s}\|g^n\|.
%\]
%Since $\{g^n\}$ is bounded in $L^2$, it follows that $\{F_n\}_{n\in\N}$ is equih\"older and accordingly equicontinuous. Moreover, $F_n(0)=0$, hence $|F_n(t)|\leq \sqrt{T}\sup_{n\in\N}\|g^n\|$ for all $t\in [0,T]$. The functions $F_n$, $n \in \N$, are therefore uniformly bounded and equicontinuous, so that $\{F_n : n \in \N\}$ has compact closure by Ascoli's theorem; every subsequence has then a uniformly converging subsequence; but
%pointwise convergence to $F$, which follows by
%\[
%	F_n(t) = \int_0^T 1_{[0,t]}(s) g^n(s)\,ds \to \int_0^T 1_{[0,t]}(s) g(s)\,ds = F(t),
%\]
%implies that all these sequences converge uniformly to $F$, which is therefore the uniform limit of the entire sequence.
\end{proof}

\begin{proof}[Hypothesis H6]
Let $\{\epsilon_n\}_{n\in\N} \subset (0,1]$ be such that $\epsilon_n\to 0$ as $n\to\infty$, and let $\{v^n\}_{n\in\N}\subset\mathcal{M}^2[0,T]$ be such that, for some constant $N>0$,
\[
	\sup_{n\in\N} \int_0^T |v^n_s(\omega)|^2\,ds < N \quad\text{for $\theta$-almost all }\omega\in\mathcal{W}^m.
\]
For $n\in\N$, let $X^{n,v_n}$ be the solution of Equation~\eqref{ControlSDE2} with $\epsilon = \epsilon_n$, control $v = v^n$, and initial condition $x$. Observe that if $\epsilon\leq 1$, then $\sqrt{\epsilon}\s_\epsilon$ has sublinear growth at infinity with constant $M$, thanks to A3. By Lemma~\ref{LemmaControlSDEGrowth} in the appendix, it follows that for all $p\geq 2$,
\begin{equation}\label{eq:estimate}
	\sup_{n\in\N} \E\left[\sup_{t\in[0,T]}|X^{n,v_n}_t|^p\right]\leq C(1+|x|^p)
\end{equation}
for some finite constant $C=C_{p}(T,N,M)$. Estimate \eqref{eq:estimate}, together with the sublinear growth at infinity of $\sigma$ (according to A1), implies in particular that
\[
	\sup_{n\in\N} \int_0^T \E\left[|\s(s,X^{n,v_n})|^2\right]\,ds <\infty.
\]
It remains to verify that the family $\{X^{n,v_n}\}_{n\in\N}$ is tight. In view of the Kolmogorov tightness criterion (for instance, Theorem~13.1.8 in \citet[pp.\,517-518]{revuzyor}), it suffices to show that there exist strictly positive constants $\alpha$, $\beta$, $\gamma$ such that for all $t,s\in [0,T]$,
\[
	\sup_{n\in\N}\E\left[\left|X^{n,v_n}_t-X^{n,v_n}_s\right|^\alpha\right] \leq \beta|t-s|^{\gamma+1}.
\]
Without loss of generality, let $s<t$. Set $K\doteq 2^{p-1}M^p(T+C(1+|x|^p))$. Exploiting the sublinear growth, we obtain for all $n\in\N$,
\begin{align*}
	\E\left[|X^{n,v_n}_t-X^{n,v_n}_s|^p\right] &\leq 3^{p-1}(t-s)^{p-1}\E\left[\int_s^t |b_{\epsilon_n}(u,X^{n,v_n})|^p\,du\right] \\
	&\quad+ 3^{p-1}\E\left[\left(\int_s^t|\s_{\epsilon_n}(u,X^{n,v_n})|\cdot|v^n|\,du\right)^p\right]\\
	&\quad+ 3^{p-1}(\epsilon_n)^\frac{p}{2} \E\left[\left|\int_s^t \s_{\epsilon_n}(u,X^{n,v_n})\,dW_u\right|^p\right]\\
	&\leq 3^{p-1}K\left((t-s)^p + N^\frac{p}{2}(t-s)^\frac{p}{2} + c_p(\epsilon_n)^\frac{p}{2}(t-s)^\frac{p}{2} \right)\\
&\leq (t-s)^\frac{p}{2}\cdot 3^{p-1}\cdot K \left(T^\frac{p}{2}+N^\frac{p}{2} + (\epsilon_n)^\frac{p}{2}\right).
\end{align*}
The hypotheses of Kolmogorov's criterion are therefore satisfied if we choose $p > 2$ and set $\alpha\doteq p$, $\beta\doteq 3^{p-1} K \left(T^\frac{p}{2}+N^\frac{p}{2} + 1\right)$, $\gamma\doteq \frac{p}{2} - 1 > 0$.
\end{proof}

\begin{proof}[Hypothesis H2 modified.]
In the proof of Theorem~\ref{ThLDP}, hypothesis H2 is only needed to show that for all $t\in[0,T]$,
\[
	\lim_{n\to\infty} \E\left[\Psi_t(X^{n,v^n},v^n)\right] = 0,
\]
where $\Psi_t:\mathcal{W}^d\times S_N\to\R$ is defined by
\[
	\Psi_t(\phi,f)\doteq \left|\phi_t-x-\int_0^t b(s,\phi)\,ds - \int_0^t\s(s,\phi)f_s\,ds\right|\wedge 1.
\]
We show that the same conclusion holds if we assume A3 and A4. Define $b^R_\epsilon\!: [0,T]\times \mathcal{W}^d \rightarrow \R^d$ by
\[
	b^R_\epsilon(s,\phi)\doteq \begin{cases}
    	b_\epsilon(s,\phi) &\text{if } \sup_{u\in[0,s]}|\phi_u|\leq R,\\
    	b_\epsilon\bigl(s,\frac{R}{\|\phi\|_\infty}\phi\bigr) &\text{otherwise.}
  \end{cases}
\]
In the same way, define $\s^R_\epsilon$, $\s^R$, and $b^R$. It is clear that the functions just defined are globally Lipschitz and bounded. Thanks to assumption A4, $b^R_\epsilon\to b^R$ and $\s^R_\epsilon\to \s^R$ uniformly on $[0,T]\times \mathcal{W}^d$. In analogy with $\Psi_t$, set
\[
	\Psi^R_t(\phi,f)\doteq \left|\phi_t-x-\int_0^t b^R(s,\phi)\,ds - \int_0^t\s^R(s,\phi)f_s\,ds\right|\wedge 1.
\]
Observe that if $\sup_{u\in[0,t]}|\phi_u|\leq R$, then $\Psi^R_t(\phi,f)=\Psi_t(\phi,f)$.
Now consider the family $\{X^{R,n}\}$ of solutions to the equation
\[
	dX^{R,n}_t=b^R_{\epsilon_n}(t,X^{R,n})\,dt+\s^R_{\epsilon_n}(t,X^{R,n})v^n_t\,dt+\sqrt{\epsilon_n}\s^R_{\epsilon_n}(t,X^{R,n})\,dW_t,
\]
with $X^{R,n}_0 = x$. The same argument as in Theorem~\ref{ThLDP} yields
\[
	\lim_{n\to\infty} \E\left[\Psi^R_t(X^{R,n},v^n)\right] = 0.
\]
For $R>0$, $n\in\N$, let $\tau_R^n$ denote the time of first exit of $X^{n,v^n}$ from the open ball of radius $R$ centered at the origin. By the locality of the stochastic integral,
\[
	\PR\left(X^{R,n}_t=X^{n,v^n}_t\text{ for all }t\leq \tau_R^n\right) = 1.
\]
On the event $\{t<\tau_R^n\}$ we have $\Psi_t(X^{n,v^n},v^n)=\Psi^R_t(X^{R,n},v^n)$. It follows that
\begin{multline} \label{eq:PsiR}
	\E\left[\Psi_t(X^{n,v^n},v^{n})\right] \\
	= \E\left[\boldsymbol{1}_{t<\tau_R^n}\cdot \Psi_t(X^{n,v^n},v^n)\right] + \E\left[\boldsymbol{1}_{t\geq\tau_R^n}\cdot \Psi_t(X^{n,v^n},v^n)\right]\\
	\leq\E\left[\Psi^R_t(X^{R,n},v^n)\right] + \PR(t\geq\tau_R^n).
\end{multline}
Using the sublinear growth condition and the estimate of Lemma~\ref{LemmaControlSDEGrowth}, we find that for all $n\in \N$,
\[
	\PR\left(t\geq\tau_R^n\right) = \PR\left(\sup_{0\leq s\leq t}|X_t^{R,n}|\geq R\right)\leq \frac{C_2(T,N,M)(1+|x|^2)}{R^2} \doteq \frac{C}{R^2}.
\]
Taking upper limits on both sides of \eqref{eq:PsiR}, we obtain
\[
	\limsup_{n\to\infty} \E\left[\Psi_t(X^{n,v^n},v^n)\right]\leq \limsup_{n\to\infty} \PR(t\geq\tau_R^n)\leq \frac{C}{R^2}.
\]
Since $R > 0$ has been chosen arbitrarily, it follows that
\[
	\lim_{n\to\infty}\E\left[\Psi_t(X^{n,v^n},v^n)\right] = 0.
\]
The job of assumption H2 is therefore carried out by A3 and A4.
\end{proof}

\begin{example}[Freidlin-Wentzell estimates] Let $\bar{b}$, $\bar{\s}$ be measurable functions from $[0,T]\times\R^d$ to $\R^d$ and $\R^{d\times m}$, respectively. Assume that $\bar{b}$, $\bar{\s}$ are locally Lipschitz continuous and satisfy a sublinear growth condition, uniformly in the time variable; that is, for every $R>0$ there exists $L_R>0$ such that for all $t\in[0,T]$, all $y,z\in \R^d$ with $|y|,|z|\leq R$,
\begin{align*}
	&|\bar{b}(t,y)-\bar{b}(t,z)|\leq L_R|y-z|,& &|\bar{\s}(t,y)-\bar{\s}(t,z)|\leq L_R|y-z|,&
\end{align*}
and there exists a constant $M>0$ such that for all $t\in[0,T]$, all $y\in \R^d$,
\begin{align*}
	&|\bar{b}(t,x)|\leq M(1+|x|),& &|\bar{\s}(t,x)|\leq M(1+|x|).&
\end{align*}
Let $X^{\epsilon}$ be the unique strong solution of the stochastic differential equation 
\[
	dX^\epsilon_t = \bar{b}(t,X^{\epsilon}_t)\,dt + \sqrt{\epsilon}\,\bar{\s}(t,X^{\epsilon}_t)\,dW_t
\]
over the time interval $[0,T]$ with initial condition $X^\epsilon_0 = x$. Set $b(t,\phi)\doteq \bar{b}(t,\phi_t)$, $\s(t,\phi)\doteq \bar{\s}(t,\phi_t)$. Then $b$, $\s$ satisfy assumptions A1--A4. By Theorem~\ref{ThLDPLip}, the family $\{X^{\epsilon}\}_{\epsilon>0}$  satisfies the large deviation principle with rate function $I_x\!: \mathcal{W}^d \rightarrow [0,\infty]$ given by
\begin{equation} \label{EqFWRateFnct}
	I_x(\phi)= \inf_{\{f\in L^2([0,T];\R^m):\phi_t =x+\int_0^t \bar{b}(s,\phi_s)\,ds+\int_0^t\bar{\s}(s,\phi_s) f_s\,ds\}} \frac{1}{2}\int_0^T |f_t|^2\,dt
\end{equation}
whenever $\{f\in L^2([0,T];\R^m):\phi_t =x+\int_0^t \bar{b}(s,\phi_s)\,ds+\int_0^t\bar{\s}(s,\phi_s) f_s\,ds\}\neq \emptyset$, and $I_x(\phi)=\infty$ otherwise.
\end{example}

\begin{remark} If $\bar{\s}$ is a square matrix such that $a(t,y)\doteq \bar{\s}(t,y)\trans{\bar{\s}(t,y)}$ is uniformly positive definite, then Equation~\eqref{EqFWRateFnct} simplifies to
\[
	I_x(\phi) =\frac{1}{2}\int_0^T \trans{(\dot{\phi_s}-\bar{b}(s,\phi_s))} a^{-1}(s,\phi_s)(\dot{\phi_s}-\bar{b}(s,\phi_s))\,ds
\]
whenever $\phi\in \mathcal{W}^d$ is absolutely continuous on $[0,T]$ with $\phi_0=x$, and $I_x(\phi) = \infty$ otherwise.
\end{remark}

%-------

\section{Two applications}
\label{SectAppl}

In Subsection~\ref{SectApplDelay}, we apply Theorem~\ref{ThLDPLip} to derive the large deviation principle for stochastic systems with memory or delay established in \citet{mohammedzhang}. They consider systems with point delay. Their proof is based on a discretization argument analogous to the method of steps for proving properties (including existence of solutions) of delay differential equations. This allows to derive the large deviation principle for It{\^o} processes with delay from the (well established) large deviation principle for It{\^o} diffusions with time dependent coefficients. The coefficients are assumed to be globally Lipschitz.

In Subsection~\ref{SectApplHoelder}, we go back to Theorem~\ref{ThLDP} to derive the large deviation principle obtained by \citet{baldicaramellino} for a class of positive It\^o diffusions with dispersion coefficient $\s$ of square-root type. In that work, as mentioned in the introduction, a general large deviation principle is established for the diffusion case (the coefficients may actually depend on the parameter $\epsilon$). The assumptions can be summarized as follows \citep[cf.][A.2.3 and Theorem~2.4]{baldicaramellino}: assumptions on $b$, $\s$ in terms of Equation~\eqref{IntegralEquation} equivalent to our hypotheses H4 and H5, including existence of solutions; local Lipschitz continuity of $b_{\epsilon}$, $\s_{\epsilon}$ for $\epsilon > 0$ as well as strong existence (and uniqueness) of solutions for the corresponding prelimit equations; the quasi-continuity property (assumption A.2.3(c) there), which relates the prelimit solutions to solutions of the limit equation \eqref{IntegralEquation}. All assumptions are verified for locally Lipschitz continuous coefficients with sublinear growth at infinity when $b_{\epsilon}$, $\s_{\epsilon}$ converge to $b$, $\sigma$ uniformly on compacts. Although the diffusion coefficient $\sigma$ in the case of positive diffusions is locally Lipschitz only on $\R\setminus\{0\}$, Proposition~\ref{PropUniqueness}, which is Proposition~3.11 in \citet{baldicaramellino}, allows to invoke the large deviation principle for locally Lipschitz coefficients. Here, we use their result only to check that uniqueness holds for the controlled deterministic limit equation (Equation~\eqref{HoelderIntEq} below).

\subsection{Systems with memory}
\label{SectApplDelay}

Let $\bar{b}\!:[0,T]\times\R^d\times\R^d\to\R^d$ and $\bar{\s}\!:[0,T]\times\R^d\times\R^d\to\R^{d\times m}$ be Borel measurable functions. Let us make the following assumptions, which are those of \citet{mohammedzhang}.
\begin{description}
  \item[Q1] The functions $\bar{b}$, $\bar{\s}$ satisfy a global Lipschitz condition; that is, there exists a constant $L>0$ such that for all $x_1,x_2,y_1,y_2\in\R^d$, all $t\in[0,T]$,
	\begin{align*}
		|\bar{b}(s,x_1,y_1)-\bar{b}(s,x_2,y_2)|&\leq L(|x_1-x_2|+|y_1-y_2|), \\
		|\bar{\s}(s,x_1,y_1)-\bar{\s}(s,x_2,y_2)|&\leq L(|x_1-x_2|+|y_1-y_2|).
	\end{align*}
  
  \item[Q2] The functions $\bar{b}(\bcdot,x,y)$, $\bar{\s}(\bcdot,x,y)$ are continuous on $[0,T]$, uniformly in $x,y\in\R^d$.
\end{description}

Let $\tau\in ]0,T[$ and $\psi \in \mathbf{C}([-\tau,0],\R^d)$; $\tau$ will be the length of the (fixed) point delay and $\psi$ the initial segment. For $\epsilon > 0$, consider the stochastic delay differential equation
\begin{equation}\label{eq:delay}
	dX^\epsilon_t = \bar{b}(t,X^\epsilon_t,X^\epsilon_{t-\tau})\,dt + \sqrt{\epsilon}\,\bar{\s}(t,X^\epsilon_t,X^\epsilon_{t-\tau})\,dW_t,
\end{equation}
over $t\in [0,T]$ and with initial condition $X^\epsilon_s=\psi_s$ for all $s\in[-\tau,0]$. Denote by $\mathcal{C}_\psi$ the set of all continuous functions $\phi\!:[-\tau,T]\rightarrow \R^d$ such that $\phi_{s}=\psi_s$ for all $s\in[-\tau,0]$. Let $G_{\psi}$ be the map $L^2([0,T],\R^m)\to \mathcal{C}_\psi$ which takes $f\in L^2([0,T],\R^m)$ to the unique solution of the integral equation
\begin{equation} \label{DelayIntEq}
	\phi_t =\begin{cases}
	\psi_0+\int_0^t\bar{b}(s,\phi_s,\phi_{s-\tau})\,ds+\int_0^t \bar{\s}(s,\phi_s,\phi_{s-\tau})f_s\,ds &\text{if } t\in]0,T],\\
	\psi_t &\text{if } t\in[-\tau,0].
	\end{cases}
\end{equation}

\begin{theorem} \label{ThLDPDelay}
Grant Q1 and Q2. Then the map $G_\psi$ is well defined and the family $\{X^\epsilon\}_{\epsilon>0}$ of solutions of the stochastic delay differential equation \eqref{eq:delay} with initial condition $X^\epsilon_s = \psi_s$ for $s\in [-\tau,0]$ satisfies the large deviation principle with good rate function $I_{\psi}\!: \mathcal{C}_\psi\rightarrow [0,\infty]$ given by
\[
	I_{\psi}(\phi)=\inf_{\{f\in L^2([0,T];\R^m): \phi = G_{\psi}(f)\}}\frac{1}{2}\int_0^T|f_t|^2\,dt
\]
whenever $\{f\in L^2([0,T];\R^m): \phi = G_{\psi}(f)\} \neq \emptyset$, $I_{\psi}(\phi) = \infty$ otherwise.
\end{theorem}

\begin{proof}
Define a function $\Phi\!: \mathcal{W}^d \rightarrow \mathcal{C}_\psi$ according to
\[
	\Phi[\phi](s)\doteq \begin{cases}
		\psi_s\cdot \boldsymbol{1}_{[-\tau,0]}(s) + \phi_s\cdot \boldsymbol{1}_{]0,T]}(s) &\text{if }\phi_0 = \psi_0,\\
		\psi_s\cdot \boldsymbol{1}_{[-\tau,0]}(s) + \psi_0\cdot \boldsymbol{1}_{]0,T]}(s) &\text{otherwise.}
	\end{cases}
\]
Define mappings $b$, $\s$ from $[0,T]\times\mathcal{W}^d$ to $\R^d$ and to $\R^{d\times m}$, respectively, according to
\begin{align*}
	b(s,\phi)&\doteq \bar{b}(s,\phi_s,\psi_{s-\tau})\cdot \boldsymbol{1}_{[0,\tau[}(s) + 
      \bar{b}(s,\phi_s,\phi_{s-\tau})\cdot \boldsymbol{1}_{[\tau,T]}(s), \\
	\s(s,\phi)&\doteq \bar{\s}(s,\phi_s,\psi_{s-\tau})\cdot \boldsymbol{1}_{[0,\tau[}(s) + 
      \bar{\s}(s,\phi_s,\phi_{s-\tau})\cdot \boldsymbol{1}_{[\tau,T]}(s),
\end{align*}
and consider the stochastic differential equation
\begin{equation}\label{eq:SDEdelay}
	dY^\epsilon_t = b(t,Y^\epsilon)\,dt + \sqrt{\epsilon}\,\s(t,Y^\epsilon)\,dW_t
\end{equation}
over $[0,T]$ with initial condition $Y^\epsilon_0 = \psi_0$. We show that the functions $b$ and $\s$ enjoy assumptions A1--A4 of Section~\ref{SectLipschitz}. Since the coefficients do not depend on $\epsilon$, it suffices to verify A1 and A2. We check the assumptions only for $b$, the work for $\s$ being completely analogous. Let us start with A1. Thanks to Q1 we have
\[
	|\bar{b}(t,x,y)|\leq L(|x|+|y|)+|\bar{b}(t,0,0)|.
\]
By Q2 it follows that $\sup_{t\in[0,T]}|\bar{b}(t,0,0)|<\infty$. Let $\phi\in\mathcal{W}^d$. Then
\begin{equation*}
	|b(s,\phi)| \leq \begin{cases}
	L(|\phi_s|+|\psi_{s-\tau}|)+|\bar{b}(s,0,0)| &\text{if } s\in[0,\tau[,\\
	L(|\phi_s|+|\phi_{s-\tau}|)+|\bar{b}(s,0,0)| &\text{if } s\in[\tau,T].
\end{cases}
\end{equation*}
Set $M \doteq 2L\vee (\sup_{t\in[0,T]}|\bar{b}(t,0,0)| + \sup_{s\in[-\tau,0]}L|\psi_s|)$. Then $|b(s,\phi)|\leq M(1+\sup_{t\in[0,s]}|\phi_t|)$, which yields A1. Next we verify A2. Let $\phi, \tilde{\phi}\in\mathcal{W}^d$. Then, thanks to Q1,
\begin{equation*}
	|b(s,\phi)-b(s,\tilde{\phi})| \leq \begin{cases}
	L|\phi_s-\tilde{\phi}_s| &\text{if } s\in [0,\tau[, \\
	L\left(|\phi_s-\tilde{\phi}_s|+|\phi_{s-\tau}-\tilde{\phi}_{s-\tau}|\right) &\text{if } s\in [\tau,T].
	\end{cases}
\end{equation*}
Thus $b(t,\bcdot)$ is globally Lipschitz continuous with constant $2L$, uniformly in $t\in[0,T]$.

Since $b$, $\s$ satisfy both A1 and A2, Theorem~\ref{ThLDPLip} applies and yields that the family $\{Y^\epsilon\}_{\epsilon>0}$ of solutions of Equation~\eqref{eq:SDEdelay} with initial condition $Y^\epsilon_0 = \psi_0$ satisfies the large deviation principle with good rate function $J\!: \mathcal{W}^d \rightarrow [0,\infty]$ given by
\[
	J(\phi) = \inf_{\{f\in L^2([0,T];\R^m):\phi = \Gamma(f)\}} \frac{1}{2}\int_0^T |f_t|^2\,dt,
\]
where $\inf\emptyset = \infty$ by convention and $\Gamma \doteq \Gamma_{\psi_{0}}$ as in H4. In particular, $\Gamma$ is well defined as the mapping $L^{2}([0,T],\R^m) \rightarrow \mathcal{W}^d$ that takes $f\in L^{2}([0,T],\R^m)$ to the unique solution of Equation~\eqref{IntegralEquation}, that is, to the unique solution $\phi\in \mathcal{W}^d$ of the integral equation
\[
	\phi_t= x + \int_0^t b(s,\phi)\,ds + \int_0^t \s(s,\phi)f_s\,ds,\quad t\in [0,T].
\]
Now let $\phi \in \mathcal{C}_{\psi}$. Then $\phi$ solves the integral equation \eqref{DelayIntEq} with $f\in L^{2}([0,T],\R^m)$ if and only if $\phi_{|[0,T]} = \Gamma(f)$. Recalling the definition of $b$, $\s$, it follows that Equation~\eqref{DelayIntEq} has a unique solution and that the mapping $G$ is well defined. Moreover, for every $\epsilon > 0$, Equation~\eqref{eq:delay} possesses a unique strong solution $X^{\epsilon}$ with initial segment $\psi$, and $X^{\epsilon} = \Phi[Y^{\epsilon}]$ $\theta$-almost surely.

Set $\boldsymbol{C}_{\psi_0}\doteq \{\phi\in \mathcal{W}^{d} : \phi_0 = \psi_0\}$. Observe that the effective domain of $J$, namely $\mathcal{D}_J \doteq \{\phi\in\mathcal{W}^d: J(\phi)<\infty\}$, is contained in $\boldsymbol{C}_{\psi_0}$. The map $\Phi$ is continuous on $\boldsymbol{C}_{\psi_0}$ (in fact a continuous bijection $\boldsymbol{C}_{\psi_0} \rightarrow \mathcal{C}_\psi$). Since the processes $Y^{\epsilon}$ take values in $\boldsymbol{C}_{\psi_0}$ and $X^{\epsilon} = \Phi[Y^{\epsilon}]$, it follows by the contraction principle (see, for instance, Theorem 4.2.1 with Remark (c) in \citet[pp.\,126-127]{dembozeitouni}) that the family $\{X^\epsilon\}_{\epsilon>0}$ satisfies the large deviation principle with good rate function $I\!: \mathcal{C}_{\psi} \rightarrow [0,\infty]$ given by
\begin{align*}
	I(\bar{\phi})&= \inf\left\{J(\phi) : \phi\in \mathcal{W}^d \text{ such that }\Phi[\phi] = \bar{\phi}\right\} \\
	&= J(\bar{\phi}_{|[0,T]})\\
	&=\inf_{\{f\in L^2([0,T];\R^m):\bar{\phi}_{|[0,T]} = \Gamma(f)\}} \frac{1}{2}\int_0^T|f_t|^2\,dt\\
	&=\inf_{\{f\in L^2([0,T];\R^m):\bar{\phi}=G(f)\}} \frac{1}{2}\int_0^T |f_t|^2\,dt.
\end{align*}
\end{proof}

\begin{remark} \label{RemDelayModels}
A closer look at the proof above shows that we can generalize without any effort the result of \citet{mohammedzhang}. We can assume the coefficients to be locally Lipschitz continuous and depend on $\epsilon$ as well, provided that a sublinear growth condition is satisfied and that $\bar{b}_\epsilon\to b$ and $\bar{\s}_\epsilon\to \s$. In particular the uniform continuity condition Q2 is no longer needed, and it suffices to assume predictability of the coefficients. With the approach used here, the large deviation analysis can be performed in the same way also for other delay models such as distributed delay or dependence on the running maximum; in those cases the coefficients could be (locally) Lipschitz functions of expressions like
\begin{align*}
	&\int_{-\tau}^{0} g(\psi_{s})ds,& &\sum_{s\in J} g_{s}(\psi_{s}),& &\max_{s\in[-\tau,0]} g(\psi_{s}),&
\end{align*}
where $g$, $g_{s}$ are suitable functions, $J\subset [-\tau,0]$ a countable set. These generalizations would be difficult to obtain with a method-of-steps approach.
\end{remark}

%-------

\subsection{Positive diffusions with H{\"o}lder dispersion coefficient}
\label{SectApplHoelder}

In this subsection, we derive the large deviation principle for a class of scalar It{\^o} diffusions where the dispersion coefficient $\s$ is positive away from zero and H{\"o}lder continuous with exponent $\gamma\geq\frac{1}{2}$. We can rely on the work by \citet{baldicaramellino} in proving uniqueness for the deterministic limit system \eqref{IntegralEquation} as required by Hypothesis H4, see Proposition~\ref{PropUniqueness} below; then we invoke Theorem~\ref{ThLDP}.

Let $W$ denote a one-dimensional Brownian motion. Slightly changing notation, let $x_{0} > 0$ be the initial condition and consider, for $\epsilon > 0$, the scalar stochastic differential equation
\begin{equation} \label{SDEHoelder}
	dX^\epsilon_t = \bar{b}(X^\epsilon_t)\,dt + \sqrt{\epsilon}\bar{\s}(X^\epsilon_t)\,dW_t
\end{equation}
with $X^\epsilon_0 = x_0$. We make the following assumptions on the coefficients $\bar{b}$, $\bar{\s}$, which we take independent of $\epsilon$ for the sake of simplicity.
\begin{description}
   \item[R1] The dispersion coefficient $\bar{\s}\!:\R\rightarrow [0,\infty)$ is locally Lipschitz continuous on $\R\setminus\{0\}$, has sublinear growth at infinity, $\bar{\s}(0) = 0$, while $\bar{\s}(x) > 0$ for all $x\neq 0$. Moreover, there exists a continuous increasing function $\rho\!:(0,\infty) \rightarrow (0,\infty)$ such that $\int_{0+}^{\infty} \rho^{-2}(u)\,du = +\infty$ and
\[
	|\bar{\s}(x)-\bar{\s}(y)|\leq \rho(|x-y|)\quad\text{for all } x,y\in\R,\; x\neq y.
\]

   \item[R2] The drift coefficient $\bar{b}:\R\to\R$ is locally Lipschitz continuous, has sublinear growth at infinity, and $\bar{b}(0) > 0$.
\end{description}

Condition R1 is satisfied, in particular, if $\bar{\sigma}(x) = \sqrt{|x|}$. The large deviation principle will be derived from Theorem \ref{ThLDP}. To this end, set
\begin{align*}
	& \s(s,\phi) \doteq \bar{\s}(\phi_s), &b(s,\phi) = \bar{b}(\phi_s), & &(s,\phi)\in [0,T]\times \mathcal{W}^1.&
\end{align*}
Let us check that hypotheses H1--H6 hold for $b$, $\s$. Since $\bar{b}$, $\bar{\s}$ are continuous, $b$, $\s$ are predictable with $b(t,\bcdot)$, $\s(t,\bcdot)$ uniformly continuous on all bounded subsets of $\mathcal{W}^1$, uniformly in $t\in [0,T]$. Moreover, given any $\phi \in \mathcal{W}^1$, $\s(\bcdot,\phi)$ is bounded by $M(1+\|\phi\|_\infty)$ for some $M$ independent of $\phi$ thanks to the sublinear growth condition, hence square-integrable. Thus H1 holds. Hypothesis H2 is clearly satisfied as $b_\epsilon \equiv b$ and $\s_\epsilon \equiv \s$. Under R1 and R2, pathwise uniqueness holds for Equation~\eqref{SDEHoelder} (or \eqref{SDE2} with $b$, $\s$ as above); this follows from Theorem~1 in \citet{yamadawatanabe}. Continuity and sublinear growth of the coefficients implies existence of a weak solution (for instance, Theorems 2.3 and 2.4 in \citet{ikedawatanabe}), which together with pathwise uniqueness actually implies that any solution is strong (Corollary~3 in \citet{yamadawatanabe} or Theorem IX.1.7 in \citet[p.\,368]{revuzyor}). Accordingly, hypothesis H3 holds. The fact that hypothesis H4 holds is a consequence of Remark~\ref{RemH4} and Proposition~\ref{PropUniqueness} stated next, which can be proved exactly as Proposition~3.11 in \citet{baldicaramellino}.

\begin{proposition} \label{PropUniqueness} Grant R1 and R2. Let $f\in L^2([0,T])$. Then uniqueness of solutions holds for the integral equation
\begin{equation} \label{HoelderIntEq}
	\phi_t = x_0+\int_0^t \bar{b}(\phi_s)\,ds+\int_0^t\bar{\s}(\phi_s)f_s\,ds.
\end{equation}
Moreover, for every $N>0$ there exists $\eta>0$ such that $\inf_{t\in[0,T]}\phi_t\geq\eta$ whenever $\phi$ is a solution of \eqref{HoelderIntEq} and $\|f\|_{L^2} < N$.
\end{proposition}

Proposition~\ref{PropUniqueness} also implies that hypothesis H5 is satisfied. The map $\Gamma_x$ which takes $f\in S_N$ to the unique solution of the integral equation
\[
\phi_t = x+\int_0^t b(\phi_s)\,ds+\int_0^t\s(\phi_s)f_s\,ds
\]
coincides with the map defined by replacing $\s$ with a function which is locally Lipschitz on the whole $\R$ and equals $\s$ outside a sufficiently small neighborhood of zero. Indeed, there exists $\xi>0$ such that, for all $f\in S_N$, $\Gamma_x(f)\geq \xi$. Therefore, $\Gamma_x$ is continuous on $S_N$ endowed with the weak topology of $L^2$, as a consequence of what we have shown in Section~\ref{SectLipschitz} in the case of locally Lipschitz continuous coefficients.

Finally, by assumptions R1 and R2, the coefficients $b$, $\sigma$ have sublinear growth at infinity. Based on this property, we can argue exactly as in Section~\ref{SectLipschitz} to show that H6 holds.

\begin{theorem} \label{ThLDPHoelder}
Grant R1 and R2. Then the family $\{X^\epsilon\}_{\epsilon>0}$ of solutions of the stochastic differential equation \eqref{SDEHoelder} with initial condition $x_0$ satisfies the large deviation principle with good rate function $I\!: \mathbf{C}([0,T],\R)\rightarrow [0,\infty]$ given by
\[
	I(\phi)= \frac{1}{2} \int_0^T \frac{(\dot{\phi}_t-\bar{b}(\phi_t))^2}{\bar{\s}^2(\phi_t)}dt
\]
whenever $\phi$ is absolutely continuous on $[0,T]$ such that $\phi_0 = x_0$ and $(\dot{\phi}-\bar{b})/\bar{\sigma}(\phi)\in L^{2}([0,T],\mathbb{R})$, and $I(\phi)=\infty$ otherwise.
\end{theorem}

\begin{proof} We have already checked that R1 and R2 imply H1--H6. Theorem~\ref{ThLDP} therefore yields the large deviation principle for the family $\{X^\epsilon\}_{\epsilon>0}$ with good rate function $J = J_{x_0}$ given by
\[
	J(\phi) = \inf_{\{f\in L^2([0,T],\R): \phi_t = x_0 + \int_0^t \bar{b}(\phi_s)\,ds +\int_0^t \bar{\s}(\phi_s)f_s\,ds\}} \frac{1}{2} \int_0^T |f_t|^2\,dt
\]
whenever $\{f\in L^2([0,T],\R): \phi_t = x_0 + \int_0^t \bar{b}(\phi_s)\,ds +\int_0^t \bar{\s}(\phi_s)f_s\,ds\} \neq \emptyset$, and $I(\phi)=\infty$ otherwise. In particular, $J(\phi) < \infty$ if and only if $\phi$ solves \eqref{HoelderIntEq} for some $f\in L^2([0,T],\R)$. Let $\phi\in \mathcal{W}^1$ be such that $J(\phi) < \infty$. Then $\phi$ solves \eqref{HoelderIntEq} for some $f\in L^2([0,T],\R)$, hence
\[
	\dot{\phi}_t = b(\phi_t)+\s(\phi_t)f_t \quad\text{for almost every } t\in [0,T],
\]
and $\phi$ is absolutely continuous on $[0,T]$ with $\phi_0 = x_0$. By Proposition~\ref{PropUniqueness}, $\phi_t>0$ for all $t\in[0,T]$, thus $\bar{\s}(\phi_t)\neq0$,
hence
\[
	\frac{\dot{\phi}_t-\bar{b}(\phi_t)}{\bar{\s}(\phi_t)} = f_t
\]
for almost every $t\in [0,T]$. It follows that
\[
\frac{1}{2}\int_0^T |f_t|^2\,dt = \frac{1}{2}\int_0^T\frac{(\dot{\phi}_t-\bar{b}(\phi_t))^2}{\bar{\s}^2(\phi_t)}dt,
\]
which implies $J(\phi) = I(\phi)$. On the other hand, if $\phi\in \mathcal{W}^1$ is absolutely continuous on $[0,T]$ with $\phi_0 = x_0$ such that $\int_0^T\frac{(\dot{\phi}_t-\bar{b}(\phi_t))^2}{\bar{\s}^2(\phi_t)}dt < \infty$, then
\[
	f_t \doteq \frac{\dot{\phi}_t-\bar{b}(\phi_t)}{\bar{\s}(\phi_t)}
\]
is well defined as an element of $L^{2}([0,T],\R)$ and $\phi$ solves \eqref{HoelderIntEq} with control $f$. It follows also in this case that $J(\phi) = I(\phi)$.
\end{proof}

%-------

\begin{appendix}

\section{Appendix}

As above, let $(\mathcal{W}^m,\mathcal{B},\theta)$ be the canonical probability space for $m$-dimensional Brownian motion over the time interval $[0,T]$, and let $(\mathcal{G}_t)$ be the $\theta$-augmented filtration generated by the coordinate process $W$. Let $\mathcal{M}^2[0,T]$ denote the space of all $\R^{m}$-valued square-integrable $(\mathcal{G}_t)$-predictable processes. Theorem~3.1 in \citet{bouedupuis} provides the following representation for Laplace functionals of the Brownian motion $W$. For all $F\!: \mathcal{W}^m \rightarrow \R$ bounded and measurable,
\begin{equation} \label{RepresentationFormulaB}
	-\log\E\left[ e^{-F(W)}\right] = \inf_{v\in \mathcal{M}^2[0,T]} \E\left[\frac{1}{2}\int_0^T|v_s|^2\,ds + F\left(W+\int_{0}^{\bcdot} v_s\,ds\right)\right],
\end{equation}
where $\E$ denotes expectation with respect to the Wiener measure $\theta$.

Let $b(\cdot,\cdot)$ and $\s(\cdot,\cdot)$ be predictable functions from $[0,T]\times\mathcal{W}^d$ to $\R^d$ and to $\R^{d\times m}$, respectively. Fix $x\in\R^d$, and consider the stochastic differential equation
\begin{equation} \label{SDE}
      dX_t=b(t,X)\,dt+\s(t,X)\,dW_t
\end{equation}
for $t\in[0,T]$ and with initial condition $X_0 = x$. Suppose that Equation~\eqref{SDE} has a strong solution. Then there exists a $\B(\mathcal{W}^m)\setminus\B(\mathcal{W}^d)$-measurable function $h\!: \mathcal{W}^m \rightarrow \mathcal{W}^d$ such that $X= h[W]$ $\theta$-almost surely; for instance, Theorem 10.4 in \citet[p.\,126]{rogerswilliams}.
Hence, for any $F\!:\mathcal{W}^d\to\R$ bounded and measurable, $F\circ h$ is a bounded and measurable map from $\mathcal{W}^m$ into $\R$. By representation formula \eqref{RepresentationFormulaB} for Brownian motion, it follows that
\begin{multline}\label{RepresentationFormulaC}
	-\log\E\left[ e^{-F(X)}\right] = -\log\E\left[e^{-F\circ h(W)}\right]\\
	= \inf_{v\in \mathcal{M}^2[0,T]} \E\left[\frac{1}{2}\int_0^T|v_s|^2\,ds + F\circ h\left(W+\int_{0}^{\bcdot} v_s\,ds\right)\right].
\end{multline}

For $v\in \mathcal{M}^2[0,T]$, consider the controlled stochastic differential equation
\begin{equation}\label{ControlSDE}
      dX^v_t = b(t,X^v)\,dt+\s(t,X^v)v_t\,dt+\s(t,X^v)\,dW_t
\end{equation}
for $t\in[0,T]$ and with initial condition $X^v_0 = x$. If strong existence and pathwise uniqueness hold for Equation~\eqref{SDE}, then the term $F\circ h\left(W+\int_{0}^{\bcdot} v_s\,ds\right)$ in \eqref{RepresentationFormulaC} can be rewritten in terms of solutions to Equation~\eqref{ControlSDE}. We only need that identity for control processes $v$ with deterministically bounded $L^2$-norm. Lemma~\ref{LemmaControlSDERep} should be compared to Theorem~4.1 in \citet{bouedupuis}.

\begin{lemma} \label{LemmaControlSDERep}
	Let $v\in \mathcal{M}^2[0,T]$ be such that $\int_{0}^{T} |v_s|^{2}ds \leq N$ $\theta$-almost surely for some $N > 0$. Suppose that strong existence and pathwise uniqueness hold for Equation~\eqref{SDE} with initial condition $X_0 = x$. Then Equation~\eqref{ControlSDE} has a unique strong solution $X^v$ with $X^{v}_0 = x$ and
\[
	h\left(W+\int_{0}^{\bcdot} v_s\,ds\right) = X^{v}\qquad \theta\text{-a.s.}
\]
\end{lemma}

\begin{proof}
Define the process
\[
	\tilde{W}_t\doteq W_t+\int_0^t v_s\,ds,\quad t\in [0,T].
\]
Since $\int_0^t |v_s|^2\,ds\leq N$ $\theta$-almost surely, Girsanov's theorem is applicable; accordingly, there exists a measure $\gamma$ over $\mathcal{W}^m$ equivalent to $\theta$ such that $\tilde{W}$ is a $(\mathcal{G}_t)$-Brownian motion on $[0,T]$ (for instance, Theorem 5.2 in \citet[p.\,191]{karatzasshreve}). With respect to the measure $\gamma$ the controlled equation \eqref{ControlSDE} becomes
\begin{equation}\label{eq:WTilde}
    dX^{v}_t=b(t,X^v)\,dt + \s(t,X^v)\,d\tilde{W}_t.
\end{equation}
Uniqueness of solutions to Equation~\eqref{ControlSDE} follows by assumption of pathwise uniqueness for Equation~\eqref{SDE}. Indeed, if $X$ and $Y$ are two solutions of the Equation~\eqref{ControlSDE} with respect to $W$ and $\theta$, then they are solutions of Equation~\eqref{eq:WTilde} with respect to $\gamma$ and $\tilde{W}$. By pathwise uniqueness, $X$, $Y$ are indistinguishable.

We now prove existence of solutions. For continuous and $(\mathcal{G}_t)$-adapted processes $Z$, define the map $\Psi(Z)\!: \mathcal{W}^m \rightarrow \mathcal{W}^d$ according to
\[
	\Psi(Z)(\omega) \doteq x + \int_{0}^{\bcdot} b(s,h[Z(\omega)])ds + \left( \int_{0}^{\bcdot} \sigma(s,h[Z(\omega)])dZ_{s}\right)(\omega).
\]
The map $\Psi(Z)$ is certainly well defined when $Z$ is given by
\[
	Z_t(\omega)\doteq\tilde{W}_t(\omega) = \omega(t) + \int_{0}^{t} v_{s}(\omega)ds
\]
with $v \in \mathcal{M}^2[0,T]$. In this situation, for $\theta$-almost all $\omega \in \mathcal{W}^m$,
\begin{multline}\label{eq:claimtool}
   \Psi(\tilde{W})(\omega) = x + \int_{0}^{\bcdot} b(s,h[\tilde{W}(\omega)])ds\\
   + \int_{0}^{\bcdot} \sigma(s,h[\tilde{W}(\omega)])v_{s}(\omega)ds + \left( \int_{0}^{\bcdot} \sigma(s,h[\tilde{W}])dW_{s}\right)(\omega),
\end{multline}
where $W$ is the coordinate process on $\mathcal{W}^m$. Since $h[W]$ is a solution of Equation~\eqref{SDE}, by construction we have
\[
	h[W(\omega)]=\Psi(W)(\omega)\quad \text{for $\theta$-almost all } \omega\in \mathcal{W}^m.
\]
By Theorem 10.4 in \citet[p.\,126]{rogerswilliams}, $h(\tilde{W})$ satisfies
\[
	h[\tilde{W}] = x + \int_{0}^{\bcdot} b(s,h[\tilde{W}])ds + \int_{0}^{\bcdot}\sigma(s,h[\tilde{W}])d\tilde{W}_s\qquad \gamma\text{-a.s.}
\]
Since $\gamma$ is equivalent to $\theta$, it follows that
\[
	h[\tilde{W}]=\Psi(\tilde{W}) \qquad\theta\text{-a.s.}
\]
Thanks to \eqref{eq:claimtool}, this implies that, $\theta$-almost surely,
\begin{align*}
	h[\tilde{W}]_t &=\Psi(\tilde{W})_t\\
	&= x + \int_{0}^{t} b(s,h[\tilde{W}])ds + \int_{0}^{t} \sigma(s,h[\tilde{W}])v_{s}ds + \int_{0}^{t} \sigma(s,h[\tilde{W}])dW_{s},
\end{align*}
showing that $h[\tilde{W}]$ is a strong solution of Equation~\eqref{ControlSDE} with respect to $W$ and $\theta$. We have already seen that pathwise uniqueness holds for Equation~\eqref{ControlSDE}. It follows that
\[
	h\left(W+\int_{0}^{\bcdot} v_s\,ds\right) = X^{v}\qquad \theta\text{-a.s.}
\]
for any solution $X^{v}$ of \eqref{ControlSDE} with $X^{v}_0 = x$. 
\end{proof}

The following lemma provides a growth estimate if the coefficients $b$, $\s$ satisfy a sublinear growth condition. The proof uses only standard arguments including localization along times of first exit, the Burkholder-Davis-Gundy inequality, and Gronwall's lemma. 

\begin{lemma} \label{LemmaControlSDEGrowth}
	Let $v\in \mathcal{M}^2[0,T]$ be such that $\int_{0}^{T} |v_s|^{2}ds \leq N$ $\theta$-almost surely for some $N > 0$. Assume that $b$, $\s$ are such that, for some $M > 0$,
\[
	|b(t,\phi)| \vee |\s(t,\phi)| \leq M\left(1 + \sup_{s\in[0,t]} |\phi_s| \right)
\]
for all $t\in [0,T]$, all $\phi\in \mathcal{W}^d$. If $X^v$ is a solution of Equation~\eqref{ControlSDE} with $X^{v}_0 = x$, then for all $p\geq2$,
\[
	\E\left[\sup_{t\in[0,T]}|X^v_t|^p\right]\leq C_p(T,N,M)\left(1+|x|^p\right),
\]
where $C_p(T,N,M)$ is non-decreasing in each of its three arguments.
\end{lemma}

\end{appendix}

%-------

%Bibliography
\bibliographystyle{abbrvnat}

\end{document}